\documentclass[a4paper, reqno,12pt]{amsart}
\usepackage{tikz, tikz-3dplot}
\usepackage{amsmath, amsthm, amssymb}

\usepackage{hyperref}

\theoremstyle{plain}
\newtheorem{te}{Theorem}[section]
\newtheorem{lem}[te]{Lemma}
\newtheorem{co}[te]{Corollary}
\newtheorem{pr}[te]{Proposition}

\newtheorem{qu}[te]{Question}
\newtheorem{con}[te]{Conjecture}
\theoremstyle{remark}
\newtheorem{re}[te]{Remark}
\newtheorem*{ack*}{Acknowledgment}

\def\n{{\bf n}}

\def\0{{\bf 0}}

\def\R{{\mathbb R}}
\def\E{{\mathbb E}}

\def\C{{\mathbb C}}
\def\S{{\mathbb S}}
\def\Z{{\mathbb Z}}
\def\P{{\mathbb P}}

\def\H{{\mathbb H}}

\def\supp{{\operatorname{supp}\,}}
\def\nint{\mathop{\diagup\kern-13.0pt\int}}

\def\les{{\;\lessapprox}\;}

\def\Nc{{\mathcal N}}
\def\Mc{{\mathcal M}}

\def\Pc{{\mathcal P}}
\def\Sc{{\mathcal S}}

\def\Uc{{\mathcal U}}

\newcommand{\norm}[1]{{\left\Vert#1\right\Vert}}

\begin{document}
	
	\title{Maximal $\Lambda(p)$-subsets of manifolds}
	
	\author{Ciprian Demeter}
	\address{Department of Mathematics, Indiana University, 831 East 3rd St., Bloomington IN 47405}
	\email{demeterc@iu.edu}
	
	\author{Hongki Jung}
	\address{Department of Mathematics, Louisiana State University, Baton Rouge, LA 70803}
	\email{hjung@lsu.edu}
	
	\author{Donggeun Ryou}
	\address{Department of Mathematics, Indiana University, 831 East 3rd St., Bloomington IN 47405}
	\email{dryou@iu.edu}

	\thanks{The first author is  partially supported by the NSF grant DMS-2349828}
	
	\begin{abstract} We construct maximal $\Lambda(p)$-subsets on a large class of curved manifolds, in an optimal range of Lebesgue exponents $p$. Our arguments combine restriction estimates and  decoupling with old and new probabilistic estimates.

	\end{abstract}

	\maketitle
	\section{Description of tight decoupling in \texorpdfstring{$\R^d$}{Rd}}

	We motivate the results in this paper with the following question. Given $d\ge 1$, $p\ge 2$ and $R\gg 1$, how many points $\xi_1,\ldots,\xi_M\in [0,1]^d$ can we find so that \begin{equation}
		\label{sqrlp}
		\frac{1}{R^d}\int_{[0,R]^d}|\sum_{n=1}^Ma_n e(\xi_n\cdot x)|^pdx\lesssim \|a_n\|_{l^2}^p
	\end{equation} 
	for each $a_n\in \C$? We use the notation $e(z)=e^{2\pi i z}$. Equivalently, 
	$$
	\int_{[0,1]^d}|\sum_{n=1}^Ma_n e(R\xi_n\cdot x)|^pdx\lesssim \|a_n\|_{l^2}^p.
	$$
	We will call $\{R\xi_n:\;1\le n\le M\}$  a (finite) $\Lambda(p)$-set. Strictly speaking, this name was reserved in the literature to the case when the points $R\xi_n$ are additionally in $\Z^d$. However, we will not enforce this restriction in our paper.

	Let us test this inequality with $a_n\equiv 1$. Constructive interference (Lemma \ref{lci}) shows that for $|x|<(100d)^{-1}$ we have
	$$|\sum_{n=1}^M e(\xi_n\cdot x)|\sim M.$$
	This shows that $M\lesssim R^{\frac{2d}{p}}$. Remarkably, constructive interference turns out to be the most severe obstruction. Moreover,  there are many choices of  points $\xi_1,\ldots,\xi_M\in [0,1]^d$ with maximal size $M\sim R^{\frac{2d}{p}}$ so that \eqref{sqrlp} holds. This is an immediate consequence of the following seminal result of Bourgain. 
	
	Given $p\ge 2$ and a finite collection $\Phi$ of uniformly bounded complex-valued functions  $\varphi_1,\ldots,\varphi_N$ on a probability space, we let $K_p(\Phi)$ be the smallest constant such that
	\begin{equation}
		\label{realvscomplex}	
		\|\sum_{n=1}^Na_n\varphi_n\|_p\le K_p(\Phi)(\sum_{n=1}^{N}|a_n|^2)^{1/2}
	\end{equation}
	holds for each $a_n\in\C$. This may be seen as a measure of the ``$L^p$ orthogonality" of $\Phi$.
	\begin{te}[\cite{Bo}]
		\label{tB}	
		For each orthonormal system $\Phi=\{\varphi_1,\ldots,\varphi_N\}$ of functions  with $\|\varphi_n\|_{L^\infty}\le 1$ and each $p>2$, there is  $\Psi\subset \Phi$ with size $\sim N^{2/p}$ such that
		$$K_p(\Psi)\lesssim 1.$$
		The implicit constant is independent of $N$. 
	\end{te} 
	It is immediate that if $\Phi$ satisfies the hypotheses of the theorem,  then $K_2(\Phi)=1$ and $K_\infty(\Phi)\le N^{1/2}$. These imply that  $1\le K_p(\Phi)\le N^{\frac12-\frac1p}$ for each $2\le p\le \infty$. The system $\Phi=\{e(nx):\;1\le n\le N\}$ on $[0,1]$ has $K_p(\Phi)\sim  N^{\frac12-\frac1p}$.
	\smallskip
	
	The $N\sim R^d$ functions $\varphi_\n(x)=e(\n\cdot x)$, $\n\in [0,R]^d\cap \Z^d$, form a uniformly bounded orthonormal system on $[0,1]^d$, so Theorem \ref{tB} is applicable. In fact, its proof  shows that a generic choice of a subset of size $N^{2/p}$ of $[0,R]^d\cap \Z^d$ will work. Such a set  will generically be somewhat evenly spread inside $[0,R]^d$ (see Remark \ref{fkroigr9ig9ti90} for the precise meaning of this, in a related context). 
	Equivalently, with high probability, the $1/R$-rescaled points $\xi_n$ will tend to be uniformly distributed in $[0,1]^d$.
	But is there a choice of $\sim R^{2d/p}$ such  points that live in a thin, highly localized set $E\subset [0,1]^d$? 
	
	A simpler version of this question was answered in \cite{D}, where the exploration of this phenomenon, termed {\em tight decoupling}, was initiated. More precisely, a rather satisfactory criterium was found for \eqref{sqrlp} to hold for constant coefficients $a_n=1$ and $\xi_n\in E$. 
	\smallskip
	
	In this paper we are concerned with the case of arbitrary coefficients. Our thin sets will be compact, smooth $m$-dimensional manifolds in $\R^d$
	\begin{equation}
		\label{ruuioregurtgurt8ohu}
		\Mc=\{(\eta,\Pi(\eta)):\;\eta\in[0,1]^m\}.
	\end{equation}
	
	Our primary goal is to prove optimal results, with $R$-independent constants. However,  we also consider similar inequalities involving arbitrarily small  $R^\epsilon$ losses. 
	Given $R$-dependent quantities $A(R),B(R)$, we write $A(R)\les B(R)$ if $A(R)\lesssim_\epsilon R^\epsilon B(R)$ as $R\to \infty$, for arbitrarily small $\epsilon>0$.
	
	In our applications, the orthonormal systems $\Phi$ will consist of complex exponentials on $[0,1]^d$, with frequencies in a set $S\subset [0,R]^d$. It is easy to see that any such $S$ may have size at most $\sim R^d$.
	
	The set $S$ may not necessarily consist of lattice points $\Z^d$. Pairwise orthogonality within $\Phi$ will be guaranteed by the weaker assumption that $s-s'$ has at least one integer coordinate for each $s,s'\in S$. If the set $S$
	is clear from the context, we will prefer writing $K_p(S)$ in place of $K_p(\Phi)$.
	\smallskip
	
	Let us formulate a few questions, that have in part been asked in \cite{D}.  
	\begin{qu}
		\label{q1}
		Consider a smooth manifold $\Mc$ in $[0,1]^d$, and let $2<p<\infty$.
		\begin{itemize}
			
			\item Q1. Is there a set $\Pc_R\subset \Mc$ for $R\gg 1$, such that $|\Pc_R|\sim R^{2d/p}$ and 
			\begin{equation}
				\label{ckljfhreughrtuifirst}
				\frac{1}{R^d}\int_{[0,R]^d}|\sum_{\xi\in\Pc_R}a_\xi e(\xi\cdot x)|^pdx\lesssim \|a_\xi\|_{l^2}^p
			\end{equation}
			holds for each  $a_\xi\in\C$? This is the same as $K_p(R\Pc_R)\lesssim 1$.
			
			\item Q2. How about if we allow $R^\epsilon$ losses,
			\begin{equation}
				\label{ckljfhreughrtui}
				\frac{1}{R^d}\int_{[0,R]^d}|\sum_{\xi\in\Pc_R}a_\xi e(\xi\cdot x)|^pdx\les \|a_\xi\|_{l^2}^p?
			\end{equation}
			\item Q3. Are there explicit examples in either case?
		\end{itemize}
	\end{qu}
	The study of tight decoupling in \cite{D} was  born partly out of the desire to explain the following phenomenon for the sphere $\S^{d-1}$.  Consider an integer $R^2$ for which $R\S^{d-1}\cap \Z^{d}$ contains $\approx R^{d-2}$ points (this is the case for infinitely many $R$). The rescaled collection $\Pc_R=\frac1{R}(R\S^{d-1}\cap \Z^d)$ is conjectured to satisfy the tight decoupling inequality \eqref{ckljfhreughrtui} with $p=\frac{2d}{d-2}$.  Due to  $l^2$ decoupling, \cite{BD}, (and periodicity),  \eqref{ckljfhreughrtui} can be seen to hold for $p\le \frac{2(d+1)}{d-1}$. This is the largest range guaranteed by $l^2$ decoupling. So how special is the example of the lattice points on the sphere, for which we have square root cancellation in the maximal range $p\le \frac{2d}{d-2}$? Our arguments will show that this behavior is not isolated, but rather generic. Tight decoupling holds with high probability.
	\smallskip
	
	Questions Q1, Q2 and Q3 only make sense for $d>1$. However, problems of similar flavor have been addressed in \cite{Bo2} in one dimension, where manifolds $\Mc$ are replaced with arithmetic sets such as the (rescaled) squares $\{\frac{n^2}{R}:\;1\le n\le \sqrt{R}\}$. The general principle that identifies $\Lambda(p)$-sets of maximal size inside a finite set $S$ can be formulated as follows. 
	\begin{pr}
		\label{form}	
		Write $N=R^d$.
		Let $S\subset [0,R]^d$ be such that $e(s\cdot x)$ form an orthonormal system on $[0,1]^d$. Assume that for some $q\ge 2$ we have
		\begin{equation}
			\label{djifureiguioteugopr6u}
			K_q(S)\lesssim |S|^{1/2}N^{-1/q}.
		\end{equation}
		Then for each $p> q$ there is a subset $S_0\subset S$ with size $|S_0|\sim N^{2/p}$ and 
		$K_p(S_0)\lesssim 1$. The same holds true for $p=q$, under the extra assumption that $K_q(S)\ge N^{\kappa}$ for some $\kappa>0$.
	\end{pr}
	This is a particular case  of Theorem \ref{mi3} below. The case $p>q$ is Corollary 5 in \cite{Bo2}, while $p=q$ is our own contribution, crucially needed to address endpoint results. 
	
	When $q=2$, this follows from  Theorem \ref{tB}. Indeed, note that  \eqref{djifureiguioteugopr6u} forces $|S|\sim N$. When $q>2$, \eqref{djifureiguioteugopr6u} can be understood as a measure of ``maximal $L^q$ orthogonality". Let us explain what this means. It first shows that $|S|\gtrsim N^{2/q}$. Since $S\subset [0,R]^d$, constructive interference shows that $	K_q(S)\gtrsim |S|^{1/2}N^{-1/q}$. Thus, subsets of $[0,R]^d$ satisfying \eqref{djifureiguioteugopr6u} have maximal possible $L^q$ orthogonality for their size. For comparison, the largest possible value of $K_q(S)$ (minimal $L^q$ orthogonality) is $\sim |S|^{1/2}|S|^{-1/q}$. Proposition \ref{form} says the following. Each $L^2$ orthonormal set  with maximal $L^q$ orthogonality contains an $L^p$ orthogonal subset of maximum possible size. 
	
	In \cite{Bo2}, Bourgain proved that if $q>4$ then
	\begin{equation}
		\label{ogithioioytuhio}	
		K_q(\{1^2,2^2,\ldots,M^2\})\lesssim M^{\frac12-\frac2q}.
	\end{equation}
	This matches \eqref{djifureiguioteugopr6u} when $R=M^2$, which leads to the existence of $\Lambda(p)$-subsets of the squares of maximal possible size, for $p>4$. 
	
	In our applications to manifolds defined by \eqref{ruuioregurtgurt8ohu}, we will use the sets $$S=\{(\n,R\;\Pi(\frac{\n}{R})):\;\n\in[0,R]^m\cap\Z^m\}.$$ 
	In Section \ref{s2} we observe that if the extension operator for $\Mc$ satisfies an $L^2\to L^q$ Fourier restriction estimate, then $S$ satisfies \eqref{djifureiguioteugopr6u} with the same $q$. This result is well-known, only our interpretation is new. Our use of Fourier restriction estimates leads to sharp results for hypersurfaces with non-zero Gaussian curvature and for non-degenerate curves. In particular, we fully characterize the range of $p$ for which Question Q1 has a positive answer for these manifolds. 
	
	The $L^2\to L^q$ restriction estimate can be seen as a measure of the ``curvature" of $\Mc$. Loosely speaking, curved manifolds cannot have ``large" intersections with hyperplanes.
	
	Bourgain's estimate \eqref{ogithioioytuhio} is a measure of the arithmetic properties on the squares, which is a substitute for curvature. The role of  hyperplanes in one dimension is played by arithmetic progressions. It is easy to show that \eqref{ogithioioytuhio} forces $\{1^2,\ldots, M^2\}$ to not contain any significantly long arithmetic progression $P$. Indeed, constructive interference shows that $K_q(P)\sim |P|^{\frac12-\frac1q}$, which leads to $|P|^{\frac12-\frac1q}\lesssim M^{\frac12-\frac2q}$. Letting $q\to 4$ shows that $|P|\lesssim_\epsilon M^\epsilon$. 
	\smallskip
	
	It perhaps comes as a surprise that $L^2\to L^q$ restriction estimates are strong enough to produce maximal $\Lambda(p)$-sets. Their proofs make a rather weak use of curvature, in particular they follow easily from decoupling estimates on the corresponding manifold. The downside of our approach in Section \ref{s2} is that all examples of $\Lambda(p)$-sets it leads to are probabilistic, so non-explicit. In Section \ref{s3} we use decoupling to produce a rather easy and elegant explicit example, albeit with $R^\epsilon$ losses. More precisely, the use of decoupling fully derandomizes the selection process by better exploiting curvature. When $p$ is even, this approach allows for deterministic input of ``flat" $\Lambda(p)$-sets, known since the work \cite{Ru} of Rudin. For such $p$, he constructed explicit subsets $S$ of $\{1,2,\ldots,N\}$ with size $\sim N^{2/p}$ and $K_p(S)\sim 1.$ No similar examples are known for other values of $p$.
	\smallskip
	
	In Section \ref{s4}  we investigate the range of smaller values of $p$, for which \eqref{ckljfhreughrtuifirst}, while false for arbitrary coefficients, was proved to hold for constant coefficients in \cite{D}. We propose an $l^p$ analog of the failing $l^2$ inequality, meant to put the results from \cite{D} into a more robust framework. However, while we cannot recover the results in \cite{D} via this more general approach, we draw some natural connections to small cap decoupling and to a conjectured $l^p$ analogue of Theorem \ref{mi3}.
	\smallskip
	
	Section \ref{prob} is devoted to refining Bourgain's probabilistic estimates. These are primarily motivated by our applications, but are also likely to find future applications.

	\begin{ack*}
		The second author is grateful to Alexander Ortiz for helpful discussions.
	\end{ack*}

	\section{A general answer for Question Q1}
	\label{s2}
	The results in this section rely on two ingredients. The first one is the following far-reaching generalization of Theorem \ref{tB}. The first part is Theorem 4 in \cite{Bo2}. The endpoint is new, see Theorem \ref{tnew} and  Remark \ref{refinal} in Section \ref{prob}.
	\begin{te}
		\label{mi3}	
		Let $\Phi$ be a finite orthonormal system consisting of complex exponentials $e(s\cdot x)$, $s\in\R^d$,  on $[0,1]^d$, and let $2\le q\le  p<\infty$.
		\\
		(a) Assume $p>q$. Then there exists $\Psi\subset \Phi$ satisfying 
		$$K_p(\Psi)\lesssim 1\text{  and  }|\Psi|\sim K_q(\Phi)^{-2q/p}|\Phi|^{q/p}.$$ 
		The implicit constants only depend on $p,q,d$, but not on $\Phi$.
		\\
		(b) $(p=q)$ Assume that there exists $\kappa>0$ such that $|\Phi|^{\kappa}  \leq K_q (\Phi)^2$. Then there exists $\Psi\subset \Phi$ satisfying 
		$$K_q(\Psi)\lesssim 1\text{  and  }|\Psi|\sim K_q(\Phi)^{-2}|\Phi|.$$
		The implicit constants only depend on $q,d,\kappa$.
	\end{te}
	Let $\Mc\subset \R^d$ be an arbitrary  $m$-dimensional manifold. We consider the orthonormal system $\Phi_{R,\Mc}$ on $[0,1]^d$ consisting of   $\varphi_\xi(x)=e(R\xi\cdot x)$ with $\xi=(\eta,\Pi(\eta))$, $\eta\in [0,1]^{m}\cap (R^{-1}\Z)^m$. Note that $|\Phi_{R,\Mc}|\sim R^m$. Constructive interference shows that for each $q\ge 2$ we have $K_q(\Phi_{R,\Mc})\gtrsim  R^{\frac{m}{2}-\frac{d}{q}}.$ The following corollary is a particular case of the general principle stated in Proposition \ref{form}. 
	\begin{co}
		\label{2.2}	
		Assume that there is $q\ge 2$ such that for each $R$ we have
		\begin{equation}
			\label{E2}
			K_q(\Phi_{R,\Mc})\sim R^{\frac{m}{2}-\frac{d}{q}}.
		\end{equation}
		Then for each $p\ge q$ there is $\Psi_R\subset \Phi_{R,\Mc}$ with $|\Psi_R|\sim R^{\frac{2d}{p}}$ and $K_p(\Psi_R)\lesssim 1.$
		
		The implicit constants may depend on $p,q,m,d$, but are independent of $R$.
	\end{co}
	\begin{proof}
		The only case that needs an explanation is $p=q$. If $\frac{m}{2}>\frac{d}{q}$, we may apply Theorem \ref{mi3} with, say, $\kappa=\frac12-\frac{d}{mq}$. If $\frac{m}{2}=\frac{d}{q}$, we may take $\Psi_R=\Phi_{R,\Mc}$.
		
	\end{proof}
	It remains to prove \eqref{E2} for appropriate $q$.
	Consider the extension operator, for continuous functions $f:[0,1]^m\to \C$ 
	$$E^{\Pi}f(x)=\int_{[0,1]^m}f(\eta)e((\eta,\Pi(\eta))\cdot x)d\eta.$$
	The following is a well-known result. We are not certain about its first appearance in the literature. 
	\begin{te}
		\label{lkgt grtithiyphip[]}
		Assume that for some $r,q\ge 1$ and each $f\in L^r$ we have 
		$$\|E^{\Pi}f\|_{L^q(\R^d)}\lesssim \|f\|_{L^r([0,1]^m)}.$$
		Let $\Pc$ be a collection of $1/R$-separated points on $\Mc$.
		
		Then for each $R\gg 1$ and each $a_\xi\in\C$ we have
		$$\|\sum_{\xi\in\Pc}a_\xi e(\xi\cdot x)\|_{L^q([0,R]^d)}\lesssim R^{\frac{m}{r'}}\|a_\xi\|_{l^r}.$$
	\end{te}
	\begin{proof}
		Define the $R^{-1}$-neighborhood of $\Mc$
		$$\Nc_\Pi(R^{-1})=\{(\eta,\Pi(\eta)+t):\:\eta\in[0,1]^m,\;t\in \R^{d-m},\;|t|<1/R\}.$$	
		We first show that for each $F$ with Fourier support inside $\Nc_{\Pi}(R^{-1})$ we have
		\begin{equation}
			\label{E1}
			\|F\|_{L^q(\R^d)}\lesssim R^{-\frac{d-m}{r'}}\|\widehat{F}\|_{L^r(\Nc_\Pi(R^{-1}))}.
		\end{equation}
		To this end, note that by Fourier inversion and Fubini
		$${F}(x)=C_d\int_{|t|<R^{-1}}[\int_{[0,1]^m}\widehat{F}(\eta,\Pi(\eta)+t)e((\eta,\Pi(\eta)+t)\cdot x)d\eta] dt.$$
		Letting $f_t(\eta)=\widehat{F}(\eta,\Pi(\eta)+t)$ we find 
		$$|{F}(x)|\le C_d\int_{B_{1/R}}|E^{\Pi}f_t(x)|dt.$$
		By Minkowski's integral inequality followed by our hypothesis applied to each $f_t$, then by H\"older's inequality on $B_{1/R}$, we have
		$$\|F\|_{L^q(\R^d)}\lesssim \int_{B_{1/R}}\|E^{\Pi}f_t(x)\|_{L^q(\R^d,dx)}dt\lesssim \int_{B_{1/R}}\|f_t\|_{L^r([0,1]^m)}dt$$
		$$\lesssim |B_{1/R}|^{1/r'}\|\widehat{F}\|_{L^r(\Nc_\Pi(R^{-1}))}\sim R^{-\frac{d-m}{r'}}\|\widehat{F}\|_{L^r(\Nc_\Pi(R^{-1}))}.$$
		Next, consider a smooth function $\phi:\R^d\to\R$ such that $\phi(x)\ge 1_{[0,1]^d}(x)$ and $\supp(\widehat{\phi})\subset [0,1/10]^d$. Then $$F(x)=\sum_{\xi\in\Pc}a_\xi\phi(x/R)e(\xi\cdot x)$$
		has Fourier support in $\Nc_{\Pi}(R^{-1})$. An easy computation using the $1/R$-separation of $\xi$ shows that
		$$\|\widehat{F}\|_{L^r}\sim R^{\frac{d}{r'}}\|a_\xi\|_{l^r}.$$
		We combine this with \eqref{E1} to write 
		$$\|\sum_{\xi\in\Pc}a_\xi e(\xi\cdot x)\|_{L^q([0,R]^d)}\lesssim \|F\|_{L^q(\R^d)}\lesssim R^{-\frac{d-m}{r'}}\|\widehat{F}\|_{L^r(\Nc_\Pi(R^{-1}))}\lesssim R^{\frac{m}{r'}}\|a_\xi\|_{l^r}.$$
	\end{proof}
	We specialize this to $r=2$, and combine it with Corollary \ref{2.2} to get the following criterium.
	\begin{pr}
		\label{4ot095gtk1}	
		Assume 	
		\begin{equation}
			\label{E3}
			\|E^{\Pi}f\|_{L^q(\R^d)}\lesssim \|f\|_{L^2([0,1]^m)}
		\end{equation}
		holds for some $q> 2$ and each continuous function  $f$. Then Question Q1 has a positive answer for each $p\ge q$. 
	\end{pr}

	Let us test the criterium with two examples.
	\smallskip
	
	\textbf{Example 1:} Assume first that  $\Mc$ is the moment curve 
	$$\Gamma^d=\{(\eta,\eta^2,\ldots,\eta^d):\;\eta\in[0,1]\}.$$
	Then inequality \eqref{E3} holds for $q\ge d(d+1)$, see e.g. \cite{Dru} for the full range of extension estimates. Thus, Q1 has a positive answer when $p\ge d(d+1)$. On the other hand, no $\Pc_R\subset \Gamma^d$ with $|\Pc_R|\sim R^{2d/p}$ can satisfy \eqref{ckljfhreughrtuifirst} when $p<d(d+1)$. To see this, we use the following general principle.   
	
	\begin{lem}(Constructive interference)
		\label{lci}	
		Let $B$ be a rectangular box in $\R^d$ with dimensions $l_1,\ldots,l_d$. Let $B^o$ be its polar box, centered at the origin and with dimensions $1/l_1,\ldots,1/l_d$. Assume $\Pc$ is a finite subset of $B$. Then
		$$|\sum_{\xi\in \Pc}e(\xi\cdot x)|\sim |\Pc|$$ 
		for $x\in\frac1{100d}B^o$.	
	\end{lem}
	To apply the lemma, we use pigeonholing to find an arc $\tau$ on $\Gamma^d$ containing at least $|\Pc_R|R^{-1/d}\sim R^{\frac{2d}{p}-\frac1d}$ points in $\Pc_R$. It is easy to see that $\tau$ is a subset of a box $B$ with dimensions roughly $R^{-1/d}$, $R^{-2/d}$,..., $R^{-1}$. 
	We have the volume estimate $|B^o\cap [0,R]^d|\sim R^{\frac{d+1}{2}}$. 
	Using the lemma we find that 
	$$\frac{1}{R^d}\int_{[0,R]^d}|\sum_{\xi\in\Pc_R\cap \tau} e(\xi\cdot x)|^pdx\gtrsim R^{2d-\frac{p}d}R^{\frac{1-d}{2}}.$$ 
	Testing \eqref{ckljfhreughrtuifirst} with $a_\xi=1_{\Pc_R\cap\tau}(\xi)$ shows that $p\ge d(d+1)$.
	\medskip
	
	\textbf{Example 2:} 
	Assume $\Mc$ is a hyper-surface ($m=d-1$) with $1\le k\le d-1$ everywhere nonzero principal curvatures. It was proved in \cite{Gre} that \eqref{E3} holds for $q\ge \frac{2(k+2)}{k}$. We specialize this to the cases of most interest.   
	\smallskip
	
	When $\Mc$ has non-zero Gaussian curvature ($k=d-1$),  Question Q1 has a positive answer when $p\ge \frac{2(d+1)}{d-1}$.  Moreover, a similar application of Lemma \ref{lci}, with $B$ a box of dimensions $R^{-1/2}\times \ldots\times R^{-1/2}\times R^{-1}$,  shows that Q1 has a negative answer when $p<\frac{2(d+1)}{d-1}$. 
	\smallskip
	
	When $\Mc$ is the cone ($\Pi(\eta)=|\eta|$, $k=d-2$),  Question Q1 has a positive answer when $p\ge \frac{2d}{d-2}$. Moreover,  Lemma \ref{lci}, with $B$ a box of dimensions $R^{-1/2}\times \ldots\times R^{-1/2}\times 1\times R^{-1}$,  shows that Q1 has a negative answer when $p<\frac{2d}{d-2}$. 
	\smallskip
	
	When $\Mc$ is a hyperplane, \eqref{E3} is false for each $q<\infty$. And indeed, Question Q1 has a negative answer for each $2<p<\infty$.
	\medskip
	
	We conclude this section with some clarifying remarks regarding the limiting  distribution of $\Pc_R$ on $\Mc$. 	Inequality \eqref{E3} may be reformulated as follows. Let $\nu$ be the measure on $\Mc$ that is the pullback of the Lebesgue measure on $[0,1]^m$. Then for each continuous $g$ on $\Mc$ we have
	\begin{equation}
		\label{equiv}	
		\|\widehat{gd\nu}\|_{L^q(\R^d)}\lesssim \|g\|_{L^2(d\nu)}.
	\end{equation}
	We have the following partial converse of Proposition \ref{4ot095gtk1}.
	\begin{pr}\label{okokooooko}
		Assume Question Q1 has a positive answer for some $p>2$. Then there is a Borel probability  measure
		$\nu$ on $\Mc$ such that \eqref{equiv} holds with $q=p$. 
	\end{pr}
	\begin{proof}
		Let $\nu$ be any weak*-limit point of the sequence of discrete measures on $\Mc$
		$$\nu_R=\frac1{|\Pc_R|}\sum_{\xi\in\Pc_R}\delta_\xi.$$
		It suffices to prove  that for each compact set $K\subset \R^d$
		$$\int_K|\widehat{gd\nu}|^p\lesssim \|g\|_{L^2(d\nu)}^p,$$
		with implicit constant independent of $g$, $K$ and $R$. Since 	 $K\subset [0,R]^d$ for large enough $R$, we have
		\begin{align*}
			\int_K|\widehat{gd\nu}|^p&=\lim_{R\to\infty} \frac{1}{|\Pc_R|^{p}}\int_K|\sum_{\xi\in\Pc_R}\widehat{g(\xi)\delta_{\xi}}|^p\sim\lim_{R\to\infty}\frac{1}{R^{2d}}\int_K|\sum_{\xi\in\Pc_R}g(\xi)e(\xi\cdot x)|^p dx\\&\le \lim_{R\to\infty}\frac{1}{R^{2d}}\int_{[0,R]^d}|\sum_{\xi\in\Pc_R}g(\xi)e(\xi\cdot x)|^p dx\lesssim \lim_{R\to\infty}\frac{1}{R^{d}}(\sum_{\xi\in\Pc_R}|g(\xi)|^2)^{p/2}\\ &\sim \lim_{R\to\infty}\|g\|_{L^2(d\nu_R)}^p=\|g\|_{L^2(d\nu)}^p.	
		\end{align*}
		
	\end{proof}
	When $\Mc$ is a hypersurface with non-zero Gaussian curvature and $p=\frac{2(d+1)}{d-1}$, each limit measure $\nu$ is absolutely continuous with respect to the surface measure on $\Mc$. This can be seen as a measure of non-concentration. Indeed, constructive interference shows that for each ball $B_r\subset \R^d$ of radius $r\ge R^{-1/2}$ we have
	$$\nu_R(B_r)\lesssim r^{d-1}.$$ 
	It follows that $\nu(B_r)\lesssim r^{d-1}$ for each $r>0$, and this shows that $\nu$ must be absolutely continuous with respect to the surface measure on $\Mc$. This is not necessarily the case for larger values of $p$. We illustrate this with the following example. The moment curve $\Gamma^3$ is contained in the hyperbolic paraboloid $\H$, see \eqref{poifuigut8hurt89}. Thus, in light of our Example 1 discussed above, tight decoupling may be achieved for $p\ge 12$ with $\Pc_R$ supported on $\Gamma^3$. All corresponding limit measures will be supported on $\Gamma^3$, thus being singular with respect to the surface measure of $\H$.
	
	The same computations show that for each $p\ge \frac{2(d+1)}{d-1}$, each limit measure (for a given hypersurface $\Mc$) satisfies 
	\begin{equation}
		\label{poeifreuf8ure8gu}
		\nu(B_r)\lesssim r^{\alpha},
	\end{equation}
	with $\alpha=\frac{2(d+1)}{p}$. Using this we observe the following. Given $0<\alpha\le d-1$ we let 
	$$C(\alpha)=\min\{q:\;\exists\; \nu\text{ probability measure on }\Mc\text { such that }\eqref{equiv}\text{ and }\eqref{poeifreuf8ure8gu}\text{ hold}\}.$$We have showed in Proposition \ref{okokooooko} that $C(\alpha)\le 2(d+1)/\alpha$. Let us prove that $C(\alpha)\ge 2(d+1)/\alpha$. The argument in Proposition 3.1 in \cite{Mi} shows that for any measure satisfying \eqref{poeifreuf8ure8gu} there is a sequence of balls $B_r$ with radii shrinking to zero, so that  $\nu(B_r)\sim r^{\alpha}.$ Let $g=1_{B_r\cap \Mc}$. Note that due to constructive interference
	$$|\widehat{gd\nu}(x)|\sim \nu(B_r\cap \Mc),$$
	for each $x\in (100d)^{-1}B$, where $B$ is a rectangular box  centered at the origin, with dimensions $r^{-1}\times\ldots\times r^{-1}\times r^{-2}$. Testing \eqref{equiv} with $g$ shows that $r^{\alpha/2}\lesssim r^{(d+1)/p}$, forcing $p\ge 2(d+1)/\alpha$. 
	
	Similar questions have been considered earlier for arbitrary measures, not necessarily supported on a given manifold. See \cite{LW} for  sharp results away from the endpoint.

	\begin{re}
		\label{fkroigr9ig9ti90}
		For each $p\ge \frac{2(d+1)}{d-1}$, a generic choice for $\Pc_R$ has all associated limit measures $\nu$ equivalent with  the surface measure on $\Mc$. This is a manifestation of the ``uniform distribution'' of $\Pc_R$, as $R\to\infty$. To prove this,  let $\Delta_R\to 0$ be such that \begin{equation}
			\label{ruioirougiotugiortu}
			\lim_{R\to\infty}\Delta_R^{d-1}e^{c\Delta_R^{d-1}|\Pc_R|}=\infty.
		\end{equation}
		Partition $\Mc$ into $\sim \Delta_R^{-(d-1)}$
		caps $\mathcal C\in\mathcal C_R$ with diameter $\sim \Delta_R$. We select the points $\Pc_R$ with equal probability $\delta_R=|\Pc_R|/R^{d-1}\sim R^{2d/p}/R^{d-1}$.
		The Chernoff-Hoeffding inequality \eqref{CH} shows that for each $\mathcal C$, 
		$$\P(|\mathcal C\cap \Pc_R|\sim \delta_R(\Delta_R^{d-1} R^{d-1}))\ge 1-e^{-c\Delta_R^{d-1}|\Pc_R|}.$$
		Combining this with \eqref{ruioirougiotugiortu} we get 
		$$\lim_{R\to \infty}\P(|\mathcal C\cap \Pc_R|\sim \delta(\Delta_R^{d-1} R^{d-1})\text{ for each }\mathcal C)=1.$$
		Due to \eqref{4rioiti5i9yi} we in fact have
		$$\P(|\mathcal C\cap \Pc_R|\sim \delta(\Delta_R^{d-1} R^{d-1})\text{ for each }\mathcal C, \text{ and } K_p(R\Pc_R)\sim 1)\sim 1,$$
		with the implicit similarity constant independent of $R$.
		It follows that $\nu_R(\mathcal C)\sim \Delta_R^{d-1}$
		for each $\mathcal C\in\mathcal C_R$. Thus, $\nu_R(B_r)\sim r^{d-1}$ for each ball $B_r$ centered on $\Mc$ with radius $1\ge r\gg \Delta_R$. Letting $R\to\infty$ we find that $\nu(B_r)\sim r^{d-1}$ for each $r\in(0,1]$.
	\end{re}

	\section{decoupling estimates; deterministic examples}
	\label{s3}
	
	The examples constructed in the previous section are non-deterministic. In this section, we introduce an alternative approach, that relies on $l^2$ decoupling. The purpose is two-fold. On the one hand, this approach will lead us to some deterministic examples, albeit with $R^\epsilon$ losses. On the other hand, $l^2$ decoupling offers a new perspective, that poses interesting challenges and motivates us to strengthen Bourgain's original probabilistic estimate. We hope that comparing the two methods will shed additional light on tight decoupling, and will stimulate further research on the topic.

	\begin{te}
		\label{t1}	
		Let $p_d= \frac{2(d+1)}{d-1}$. Consider a smooth compact hypersurface $\Sigma \in \R^d$
		$$\Sigma=\{(\eta,\Pi(\eta)):\;\eta\in[0,1]^{d-1}\},$$ 
		with positive principal  curvatures (positive second fundamental form). Then for each  $R\gg 1$  there is a set $\Pc_R\subset \Sigma$ with cardinality $\sim R^{2d/p_d}$  such that 
		\begin{equation}
			\label{ehyp}
			\frac{1}{R^d}\int_{[0,R]^d}|\sum_{\xi\in\Pc_R}a_\xi e(\xi\cdot x)|^{p_d}dx\les \|a_\xi\|_{l^2}^{p_d}
		\end{equation}
		holds for arbitrary  $a_\xi\in\C$.
	\end{te}

	We will use two ingredients.  
	\begin{lem}
		\label{l1}	
		Let $p\ge 2$. Each set $\Pc$ of  $N$ points in $(R^{-1}\Z)^{d-1}$ contains a subset $S$ with $|S|\sim N^{2/p}$ such that	
		\begin{equation}
			\label{e1}
			\frac{1}{R^d}\int_{[0,R]^d}|\sum_{\eta\in S}a_\eta e((\eta,\Pi(\eta))\cdot x)|^pdx\lesssim \|a_\eta\|_{l^2}^p\end{equation}
		holds for arbitrary $[0,R]^d$ and $a_\eta\in\C$.	The implicit constant  is independent of $R$, $N$, $a_\eta$.
	\end{lem}\begin{proof}
		Write $x=(\bar{x},x_d)\in\R^{d-1}\times \R$.	
		The functions $(e(R\eta\cdot\bar{x}))_{\eta\in\Pc}$ form an orthonormal system on $[0,1]^{d-1}$. 	
		We use Theorem \ref{tB} to find $S$ with $|S|\sim N^{2/p}$ such that for each $b_\eta\in\C$ 
		$$\frac{1}{R^{d-1}}\int_{[0,R]^{d-1}}|\sum_{\eta\in S}b_\eta e(\eta\cdot \bar{x})|^pd\bar{x}\lesssim \|b_\eta\|_{l^2}^p.$$
		We then use this inequality with $b_\eta=a_\eta e(\Pi(\eta)x_d)$ and integrate over $x_d\in[0,R]$. This gives \eqref{e1}.
		
	\end{proof}
	We  also need the $l^2$ decoupling for $\Sigma$ proved in \cite{BD}.
	\begin{te}
		\label{tbd}	
		Let $R\gg 1$. 	
		Consider a partition of $\Sigma$ into square-like caps $\theta$ with diameter $\sim R^{-1/2}$. Given a (complex) measure $\mu$ supported on $\Sigma$ we denote by $\mu_\theta$ its restriction to $\theta$.
		
		We have
		$$\|\widehat{\mu}\|_{L^{\frac{2(d+1)}{d-1}}([0,R]^d)}\les (\sum_{\theta}\|\widehat{\mu_\theta}\|_{L^{\frac{2(d+1)}{d-1}}([0,R]^d)}^2)^{1/2}.$$
	\end{te}
	
	In our application, $\mu$ will be a linear combination of Dirac deltas.
	\begin{proof}(of Theorem \ref{t1})

		Let $\Theta$ be a partition of $[0,1]^{d-1}$ into cubes $\theta$ with side length $\sim R^{-1/2}$. Since $|\theta\cap (R^{-1}\Z)^{d-1}|\sim R^{\frac{d-1}{2}}$,  we may use Lemma \ref{l1} to create a set $S_\theta\subset \theta$ with size $\sim (R^{\frac{d-1}{2}})^{2/p_d}$
		such that, writing $\Pc_R(\theta)=\{(\eta,\Pi(\eta)):\;\eta\in S_\theta\}$, we have for each $b_\xi$
		\begin{equation}
			\label{e2}
			\frac{1}{R^d}\int_{[0,R]^d}|\sum_{\xi\in \Pc_R(\theta)}b_\xi e(\xi\cdot x)|^{p_d}dx\lesssim \|b_\xi\|_{l^2(\Pc_R(\theta))}^{p_d}.
		\end{equation}
		We let $$\Pc_R=\bigcup_{\theta\in\Theta}\Pc_R(\theta).$$
		Note first that our choice for $p_d$ forces $$|\Pc_R|\sim |\Theta|R^{\frac{d-1}{p_d}}\sim R^{\frac{2d}{p_d}}.$$ Fix $a_\xi$. Then  Theorem \ref{tbd} implies that
		\begin{equation}
			\label{e3}
			\int_{[0,R]^d}|\sum_{\xi\in\Pc_R}a_\xi e(\xi \cdot x)|^{p_d}dx\les \left(\sum_{\theta\in\Theta}(\int_{[0,R]^d}|\sum_{\xi\in \Pc_R(\theta)}a_\xi e(\xi\cdot x)|^{p_d}dx)^{2/p_d}\right)^{p_d/2}.
		\end{equation}	
		Now \eqref{ehyp} follows by first applying \eqref{e3}, then \eqref{e2} with $b_\xi=a_\xi$.
		
		This argument relies crucially on the fact that $p=p_d$ in two ways. First, this guarantees the right size for $\Pc_R$. Second, \eqref{e3} is no longer true for $p>p_d$.
		
	\end{proof}
	The critical index $p_d$ is an even integer when $d=2$ and $d=3$. For all even integers $p$, Rudin \cite{Ru} has constructed explicit examples of $\Lambda(p)$ sets with maximal size. To keep things simple and self-contained, and since our use of decoupling inevitably introduces $R^\epsilon$-losses, we will instead use an elegant, explicit construction of an almost $\Lambda(4)$-set of maximal size in two dimensions. 
	
	\begin{lem}
		We have for each $a_{n,m}\in \C$ and each $\epsilon>0$
		$$\|\sum_{n,m=1}^{N}a_{n,m}e(n^2x+m^2y)\|_{L^4([0,1]^2)}\lesssim_{\epsilon}N^{\epsilon}\|a_{n,m}\|_{l^2}.$$
	\end{lem}  
	\begin{proof}
		This is a well-known argument. 	
		We raise to the forth power, and analyze the 4-linear expression. The only contributing terms are those for which 
		$$\begin{cases}n_1^2+n_2^2=n_3^2+n_4^2\\m_1^2+m_2^2=m_3^2+m_4^2\end{cases}.$$
		Given $n_1,n_2,m_1,m_2$, the values of $n_3,n_4,m_3,m_4$ are determined up to $N^\epsilon$ many choices, due to the classical divisor bound. 
		
	\end{proof}
	\begin{co}
		Consider a smooth compact hypersurface in $\R^3$
		$$\Sigma=\{(\eta,\Pi(\eta)):\;\eta\in[0,1]^{2}\}$$ 
		with positive principal  curvatures. For each $R\gg 1$ consider the subset of $\Sigma$
		$$\Pc_R=\{(\eta,\Pi(\eta)),\;\eta=(\frac{i}{\sqrt{R}}+\frac{n^2}{R}, \frac{j}{\sqrt{R}}+\frac{m^2}{R}):\;0\le i,j<\sqrt{R}-1,\;1\le n,m<R^{1/4} \}$$
		with size $|\Pc_R|\sim R^{3/2}$. Then 
		$$
		\frac{1}{R^3}\int_{[0,R]^3}|\sum_{\xi\in\Pc_R}a_\xi e(\xi\cdot x)|^{4}dx\les \|a_\xi\|_{l^2}^{4}
		$$
		holds for arbitrary  $a_\xi\in\C$.
	\end{co}
	The best known examples of such $\Sigma$ are the sphere and the elliptic paraboloid. 
	
	An interesting contrast is provided by the hyperbolic paraboloid in $\R^3$
	\begin{equation}
		\label{poifuigut8hurt89}
		\H=\{(\eta_1,\eta_2,\eta_1\eta_2):\;(\eta_1,\eta_2)\in[0,1]^2\}.
	\end{equation}
	Its two principal curvatures have opposite signs.
	The method in this section fails to produce a deterministic example of an almost $\Lambda(4)$-set $\Pc_R\subset \H$ of maximal size $\sim R^{3/2}$, and we will shortly see why. Nevertheless, we are able to produce non-explicit examples. The argument combines an interesting decoupling recently proved in \cite{GMO} with a refinement of Bourgain's probabilistic estimate, which is of independent interest.  
	
	\begin{te}
		\label{last}	
		For each  $R\gg 1$  there is a set $\Pc_R\subset \H$ with cardinality $\sim R^{3/2}$  such that 
		$$
		\frac{1}{R^3}\int_{[0,R]^3}|\sum_{\xi\in\Pc_R}a_\xi e(\xi\cdot x)|^4dx\les \|a_\xi\|_{l^2}^4
		$$
		holds for arbitrary  $a_\xi\in\C$.
	\end{te}

	It was observed in \cite{BD2} that $l^2(L^p)$ decoupling into square-like caps $\theta$ (the analog of Theorem \ref{tbd}) is false for each $p>2$. The reason is that $\H$ contains lines. Each measure $\mu$ supported on such a line would fail $l^2$ decoupling, due to the lack of curvature.
	
	In fact, through each point on $\H$ there are two lines on $\H$. Their projections onto the $(\eta_1,\eta_2)$ plane are parallel to either the $\eta_1$ or the $\eta_2$ axis. Remarkably, there is an almost partition of $\H$ which leads to a successful $l^2$ decoupling on the critical space $L^4(\R^3)$. Let us describe it.
	
	For each dyadic number $R^{-1}\le 2^l\le 1$, we consider a partition $\Uc_l$ of $[0,1]^{2}$ into axis-parallel rectangles $U$ with dimensions roughly $(2^l,R^{-1}2^{-l})$. We call $\Sc_l$ the collection of vertical projections $S$ of these rectangles onto $\H$, and write $\Sc=\cup_{l}\Sc_l$. While $\Sc$ is not a partition, it is immediate that each $\xi\in\H$ is contained in only $\sim \log R$ many $S\in\Sc$.
	\medskip

	The following is Theorem 2.5 in \cite{GMO}. It will be one of the main ingredients in our proof of Theorem \ref{last}.
	\begin{te}
		\label{gmo}	
		Let $R\gg 1$. 	
		Given a (complex) measure $\mu$ supported on $\H$ we denote by $\mu_S$ its restriction to the set $S\in \Sc$.
		
		We have
		$$\|\widehat{\mu}\|_{L^{4}([0,R]^3)}\les (\sum_{S\in\Sc}\|\widehat{\mu_S}\|_{L^{4}([0,R]^3)}^2)^{1/2}.$$
	\end{te}
	The following proposition records the second ingredient needed in the proof of Theorem \ref{last}. Since the sets in $\Sc$ have some overlap, the selection of one exponential $e(\xi\cdot x)$ will affect the $K_p$ constant of all the sets in $\Sc$ containing $\xi$. Addressing this needs some extra care.  
	\begin{pr}
		\label{p3}	
		Let $p>2$.	
		Consider an orthonormal system $\Phi$ consisting of $M$ complex exponentials  $\varphi_1,\ldots,\varphi_M$ on $[0,1]^d$. For each subset $S$ of $\{1,2,\ldots,M\}$ we denote by  $\Phi_S$  the set $(\varphi_i)_{i\in S}$.
		
		Let $\Sc$ be a collection of subsets of $\{1,\ldots,M\}$, each having cardinality $N$. Assume that, say, $|\Sc|\le N^2$. Then there is a subset $G\subset\{1,\ldots,M\}$ with cardinality $\sim MN^{\frac2p-1}$ such that 
		$$K_p(\Phi_{G\cap S})\lesssim \log N$$
		for each $S\in\Sc$. 
	\end{pr}
	\begin{proof}
		We consider $\{0,1\}$-valued i.i.d.'s $\delta_i$, $1\le i\le M$, with $\P(\delta_i(\omega)=1)=N^{\frac2p-1}$. We let $G(\omega)=\{i:\;\delta_i(\omega)=1\}$ and note the expected value $\E_\omega(|G(\omega)|)=MN^{\frac2p-1}$.
		Then we have (see \eqref{CH})
		\begin{equation}
			\label{jriufreuiorugio}
			\P(|G(\omega)|\sim MN^{\frac2p-1})>\frac12.
		\end{equation}
		
		On the other hand, Theorem \ref{mir2} (a) with $q=2$ implies that $\P(K_p(\Phi_{G(\omega)\cap S})\ge \log N)\lesssim N^{-C_p\log N}$  for each $S\in\Sc$ and some $C_p>0$. Thus,
		$$\P(K_s(\Phi_{G(\omega)\cap S})\le \log N,\;\forall S\in\Sc)\gtrsim 1-|\Sc|N^{-C_p\log N}>\frac12.$$
		Combining this with \eqref{jriufreuiorugio} leads to the desired conclusion.
		
	\end{proof}
	\begin{proof}(of Theorem \ref{last})
		
		We consider the orthonormal system $\Phi$ on $[0,1]^3$ consisting of the functions $$\varphi_{n,m}(x)=e(R(\frac{n}{R},\frac{m}R,\frac{nm}{R^2})\cdot x),\; 1\le n,m\le R.$$ Proposition \ref{p3} applied to $\Phi$ with $p=4$, $M=R^2$, $N=R$ delivers a set $G\subset \{1,\ldots,R\}^2$ with cardinality $\sim R^{3/2}$ such that
		\begin{equation}
			\label{edkui95tu95}
			K_4(\Phi_{G\cap U})\le \log R\end{equation}
		for each $U\in \cup_{l}\Uc_l$. The $\log R$ loss is harmless, as it is absorbed in the notation $\les$.

		We define $$\Pc_R=\{(\frac{n}{R},\frac{m}{R},\frac{nm}{R^2}):\;(n,m)\in G\}.$$
		Then \eqref{edkui95tu95} means that for each $S\in\Sc$ and $a_\xi$
		$$\frac{1}{R^3}\int_{[0,R]^3}|\sum_{\xi\in \Pc_R\cap S}a_{\xi}e(\xi\cdot x)|^4dx\le \|a_{\xi}\|_{l^2(\Pc_R\cap S)}^4.$$
		On the other hand, Theorem \ref{gmo} implies that 
		$$\int_{[0,R]^3}|\sum_{\xi\in \Pc_R}a_{\xi}e(\xi\cdot x)|^4dx\les \left(\sum_{S\in\Sc}(\int_{[0,R]^3}|\sum_{\xi\in \Pc_R\cap S}a_{\xi}e(\xi\cdot x)|^4dx)^{1/2}\right)^2.$$
		We combine the last two inequalities, and use the fact that each $\xi\in\Pc_R$ appears in $\sim \log R$ many sets $S$, to conclude that
		$$\frac1{R^3}\int_{[0,R]^3}|\sum_{\xi\in \Pc_R}a_{\xi}e(\xi\cdot x)|^4dx\les\|a_\xi\|_{l^2}^4.$$
	\end{proof}
	It would be interesting to find an explicit set $\Pc_R$ satisfying  Theorem \ref{last}. Also, rather intriguingly, the most natural analog of Theorem \ref{gmo} for higher dimensional hyperbolic paraboloids is false, see Theorem B.1 in \cite{GMO}. It nevertheless seems reasonable to ask whether there is a decoupling-type approach to the existence of $\Lambda(p)$-sets on these manifolds.
	\medskip
	
	A similar question may be asked for the moment curve $\Gamma^d$ at the critical exponent $\gamma_d=d(d+1)$, for $d\ge 3$. Even though $\gamma_d$ is an  even integer, we do not know how to construct explicit examples. In order to help the reader understand the difficulty of this task, we present a decoupling proof of the following theorem.	
	\begin{te}
		\label{t2}
		For each $p\ge p_d$ and $R\gg 1$ there is a set $\Pc_R\subset \Gamma^d$ with $\sim R^{2d/p}$ points satisfying
		\begin{equation}
			\label{e4}
			\frac{1}{R^d}\int_{[0,R]^d}|\sum_{\xi\in\Pc_R}a_\xi e(\xi\cdot x)|^pdx\les \|a_\xi\|_{l^2}^p.
		\end{equation}
	\end{te}	
	
	We will use the following $l^2$ decoupling for the moment curve proved in \cite{BDG}. 
	\begin{te}
		\label{tbdg}	
		Let $R\gg 1$. 	
		Consider a partition of $\Gamma^d$ into arcs $\tau$ with length $\sim R^{-1/d}$. Given a (complex) measure $\mu$ supported on $\Gamma^d$ we denote by $\mu_\tau$ its restriction to $\tau$.
		
		We have
		$$\|\widehat{\mu}\|_{L^{\gamma_d}([0,R]^d)}\les (\sum_{\tau}\|\widehat{\mu_\tau}\|_{L^{\gamma_d}([0,R]^d)}^2)^{1/2}.$$
	\end{te}
	The key step is the following proposition.
	\begin{pr}
		\label{p1}	
		Each arc $\tau\subset \Gamma^d$ with length $\sim R^{-1/d}$ contains a set $\Pc_R(\tau)$ with cardinality $\sim R^{\frac{2d}{\gamma_d}-\frac1d}$ such that for each $a_\xi\in\C$
		\begin{equation}
			\label{e7}
			\frac{1}{R^d}\int_{[0,R]^d}|\sum_{\xi\in\Pc_R(\tau)} a_\xi e(\xi\cdot x)|^{\gamma_d}dx\les \|a_\xi\|_{l^2}^{\gamma_d}.
		\end{equation}
		Moreover the points in $\Pc_R(\tau)$ have the first coordinate in the set $R^{-1}\Z$.
	\end{pr}
	\begin{proof}
		Let us denote by $\pi_d:\Gamma^{d+1}\to\Gamma^d$ the bijection
		$$\pi_d(t,\ldots,t^{d+1})=(t,\ldots,t^d).$$	
		We use induction on $d$. The case $d=2$ follows by using Theorem \ref{tB}, as in the proof of Theorem \ref{t1}. Indeed, there are $N\sim R^{1/2}$ points on $\tau$ with first coordinate in $R^{-1}\Z$. It suffices to note that $N^{\frac2{\gamma_2}}\sim R^{\frac{4}{\gamma_2}-\frac12}$.
		
		Assume we have verified the conclusion for some $d$, and let us prove it for $d+1$. Consider $\tau\subset \Gamma^{d+1}$ with length $\sim R^{-\frac1{d+1}}$. We partition $\tau$ into arcs $\tau'$ with length $\sim R^{-\frac1d}$. 
		\\
		\\
		Step 1.
		The arc  $\pi_d(\tau')$  on $\Gamma^d$ has length $\sim R^{-\frac1{d}}$.  Using our induction hypothesis and Fubini we may find $\Pc_R(\tau')\subset \tau'$ with cardinality $\sim R^{\frac{2d}{\gamma_d}-\frac1d}$ such that for each $a_\xi$
		\begin{equation}
			\label{e6}
			\frac{1}{R^{d+1}}\int_{[0,R]^{d+1}}|\sum_{\xi\in\Pc_R(\tau')} a_\xi e(\xi\cdot x)|^{\gamma_d}dx\les \|a_\xi\|_{l^2}^{\gamma_d}.
		\end{equation}
		More precisely, the set $\Pc_R(\tau')$ is the preimage $\pi_d^{-1}$ of the set on $\pi_d(\tau')$ guaranteed by the induction hypothesis. 
		Let $$\Sc_R(\tau)=\bigcup_{\tau'\subset \tau}\Pc_R(\tau').$$
		Note first that $$|\Sc_R(\tau)|\sim R^{\frac1{d}-\frac1{d+1}}R^{\frac{2d}{\gamma_d}-\frac1d}.$$
		Also, the points in $\Sc_R(\tau)$ have the first coordinate in  $R^{-1}\Z$.
		\\
		\\
		Step 2.  We use Theorem \ref{tbdg} to separate the contribution from $\tau$ into the contributions from various $\tau'$. Via Fubini, this naturally extends to $l^2(L^{\gamma_d})$ decoupling (sometimes referred to as {\em cylindrical decoupling}) for $\Gamma^{d+1}$, since $\pi_d(\Gamma^{d+1})=\Gamma^d$.
		We may thus write
		$$\int_{[0,R]^{d+1}}|\sum_{\xi\in\Sc_R(\tau)} a_\xi e(\xi\cdot x)|^{\gamma_d}dx\les\left(\sum_{\tau'\subset \tau}(\int_{[0,R]^{d+1}}|\sum_{\xi\in\Pc_R(\tau')} a_\xi e(\xi\cdot x)|^{\gamma_d}dx)^{2/\gamma_d}\right)^{\gamma_d/2}.$$
		This together with \eqref{e6} implies that 
		$$\frac1{R^{d+1}}\int_{[0,R]^{d+1}}|\sum_{\xi\in\Sc_R(\tau)} a_\xi e(\xi\cdot x)|^{\gamma_d}dx\les\|a_\xi\|_{l^2}^{\gamma_d}.$$
		\\
		\\
		Step 3.  We apply Theorem \ref{mi3} to the  orthonormal system $\Phi(\tau)$ consisting of $e(R\xi\cdot x)$ with $\xi\in\Sc_R(\tau)$.  We use  the conclusion of Step 2 asserting that $K_{\gamma_d}(\Phi(\tau))\les 1$. Theorem \ref{mi3} with $q=\gamma_d$ and $p=\gamma_{d+1}$ delivers a subset $\Phi^*(\tau)\subset \Phi(\tau)$ with size $\sim |\Sc_R(\tau)|^{\gamma_d/\gamma_{d+1}}$ and $K_{\gamma_{d+1}}(\Phi^*(\tau))\les 1$. We call $\Pc_R(\tau)$ the frequencies corresponding to $\Phi^*(\tau)$. An easy computation shows that $|\Pc_R(\tau)|\sim R^{\frac{2(d+1)}{\gamma_{d+1}}-\frac1{d+1}}$, as desired.
		
	\end{proof}
	\begin{proof}(of Theorem \ref{t2})
		We first consider the case $p=\gamma_d$.	
		Consider a partition of $\Gamma^d$ into arcs $\tau$ of length $\sim R^{-1/d}$. Let $\Pc_R(\tau)$ be as in Proposition \ref{p1}, and write
		$\Pc_R=\bigcup_{\tau}\Pc_R(\tau)$. Note that $|\Pc_R|\sim R^{\frac{2d}{\gamma_d}}$, as desired. Combining \eqref{e7} with $l^2(L^{\gamma_d})$ decoupling for $\Gamma^d$ verifies \eqref{e4} for $\gamma_d$.
		
		The validity of \eqref{e4} for $p>\gamma_d$ follows by invoking again Theorem \ref{mi3}.
		
	\end{proof}		
	
	\section{$l^p$ analogs}
	\label{s4}
	To address the range of $p$ for which the $K_p$ constants cannot be $\les 1$, we replace the $l^2$ norm with the $l^p$ norm. We illustrate this with two more questions.
	\begin{qu}
		\label{q2}
		Consider a smooth manifold $\Mc$ in $[0,1]^d$, and $p\ge 2$. 
		\begin{itemize}
			
			\item Q4. Is there a set $\Pc_R\subset \Mc$ for $R\gg 1$, such that $|\Pc_R|\sim R^{2d/p}$ and 
			\begin{equation}
				\label{ckljfhreughrtuifirsthkktopptohi}
				\frac{1}{R^d}\int_{[0,R]^d}|\sum_{\xi\in\Pc_R}a_\xi e(\xi\cdot x)|^pdx\lesssim |\Pc_R|^{\frac{p}{2}-1}\|a_\xi\|_{l^p}^p
			\end{equation}
			holds for each  $a_\xi\in\C$?
			
			\item Q5. How about if we allow $R^\epsilon$ losses
			\begin{equation}
				\label{ckljfhreughrtui0h0h0-}
				\frac{1}{R^d}\int_{[0,R]^d}|\sum_{\xi\in\Pc_R}a_\xi e(\xi\cdot x)|^pdx\les |\Pc_R|^{\frac{p}2-1}\|a_\xi\|_{l^p}^p.
			\end{equation}
		\end{itemize}
	\end{qu}
	Given $p\ge 2$ and a collection $\Phi$ of functions $\varphi_1,\ldots,\varphi_N$, we let $K_p^{*}(\Phi)$ be the smallest constant for which
	$$\|\sum_{n=1}^Na_n\varphi_n\|_p\le K_p^*(\Phi)N^{\frac12-\frac1p}(\sum_{n=1}^{N}|a_n|^p)^{1/p}$$
	holds for each $a_n\in\C$. Note that $K_p^*(\Phi)\le K_p(\Phi)$, so our earlier results for Question Q1 provide answers to Question Q4 in a certain range of $p$.
	\smallskip
	
	We make the following conjecture. It is analogous to part (a) of Theorem \ref{mi3}. While part of the proof of Theorem \ref{mi3} translates naturally to the $l^p$-setting, there are critical steps that do not.
	\begin{con}
		\label{mir4}	
		Let $2\le q< p<\infty$ and let $\Phi$ be a finite, uniformly bounded orthonormal system. Then there exists $\Psi\subset \Phi$ satisfying 
		$$K_p^*(\Psi)\lesssim 1\text{  and  }|\Psi|\sim K_q^*(\Phi)^{-2q/p}|\Phi|^{q/p}.$$   	
	\end{con}
	We have the following analog of Proposition \ref{4ot095gtk1}.
	\begin{pr}
		\label{4ot095gtk2}	
		Assume 	
		\begin{equation}
			\label{E4}
			\|E^{\Pi}f\|_{L^q(\R^d)}\lesssim \|f\|_{L^q([0,1]^m)}
		\end{equation}
		holds for some $q\ge 2$. Then, conditional to Conjecture \ref{mir4}, Question Q4 has a positive answer for each $p>q$. 
	\end{pr} 
	\begin{proof}
		Theorem \ref{lkgt grtithiyphip[]} with $r=q$ implies that 
		$$K_q^*(\Phi_{R,\Mc})\lesssim R^{\frac{m}{2}-\frac{d}{q}}.$$
		The result follows by combining this estimate with Conjecture \ref{mir4}.
	\end{proof}
	Let us examine the previous examples. 
	\smallskip
	
	\textbf{Example 1}: When $\Mc=\Gamma^d$, \eqref{E4} is known to hold for $q>\frac{d^2+d+2}{2}$. See \cite{Dru}. Thus, conditional to Conjecture \ref{mir4}, Question Q4 has a positive answer in the range $p>\frac{d^2+d+2}{2}$. This was verified in \cite{D} for constant coefficients. On the other hand, it was proved in \cite{D} that the answer is negative if $p<\frac{d^2+d+2}{2}$. Thus, apart from the critical index, the question is (conditionally) settled.
	
	\smallskip
	
	\textbf{Example 2:} When $\Mc$ is a hypersurface with non-zero Gaussian curvature,
	\eqref{E4} is conjectured to hold in the range $q>\frac{2d}{d-1}$. This has only been verified when $d=2$. 
	Thus, we expect Question Q4 to have a positive answer in the range $p>\frac{2d}{d-1}$. For constant coefficients this was confirmed in \cite{D}.
	Also, the result in \cite{D} together with Theorem 1 in \cite{AN} shows that Question Q4 has a negative answer when $p\le \frac{2d}{d-1}$.
	\medskip
	
	We present some partial unconditional answers to Question Q5, for low dimensional curves. In particular, we fully settle Question Q5 for the parabola.
	\begin{te}
		For each $4\le p<6$	there is $\Pc_R\subset \Gamma^2$ with $|\Pc_R|\sim R^{4/p}$,  such that \eqref{ckljfhreughrtui0h0h0-} holds.
		
		For each $10\le p<12$ there is $\Pc_R\subset \Gamma^3$ with $|\Pc_R|\sim R^{6/p}$,  such that \eqref{ckljfhreughrtui0h0h0-} holds.
	\end{te}
	\begin{proof}
		We first prove the result for $\Gamma^2$. Partition the parabola into arcs $\gamma$ of length $\sim R^{-\beta}$ with $\beta=\frac{2}{p-2}$. Note that $\beta\in(\frac12,1]$. Theorem \ref{tB} guaranteed the existence of $\Pc_R(\gamma)\subset \gamma$ with size $\sim R^{(1-\beta)\frac2p}$ such that
		\begin{equation}
			\label{rokfitugijuhi}
			\frac1{R^2}\int_{[0,R]^2}|\sum_{\xi\in\Pc_R(\gamma)}a_\xi e(\xi\cdot x)|^pdx\lesssim \|a_\xi\|_{l^2}^p\le |\Pc_R(\gamma)|^{\frac{p}{2}-1}\|a_\xi\|_{l^p}^p.
		\end{equation}
		Then $\Pc_R=\cup_\gamma\Pc_R(\gamma)$ has the right size.
		
		The key new ingredient  is the small cap decoupling for the parabola proved in \cite{DGW}, that we  recall below. 
		
		For each $\beta\in(\frac12,1]$ and each complex measure $\mu$ supported on $\Gamma^2$ we have
		$$\|\widehat{\mu}\|_{L^p([0,R]^2)}^p\les R^{\beta(\frac{p}{2}-1)}\sum_{\gamma}\|\widehat{\mu_\gamma}\|^p_{L^p([0,R]^2)}.$$
		When applied to $\mu=\sum_{\xi\in\Pc_R}a_\xi\delta_\xi$ and combined with \eqref{rokfitugijuhi}, it finishes the proof in this case. 
		\smallskip	
		
		We next focus on $\Gamma^3$. Let $\beta=\frac{2}{p-6}$. Note that $\beta\in(\frac13,\frac12]$.	
		The argument in Section \ref{s3}  shows that each arc $\gamma\subset \Gamma^3$ of length $\sim R^{-\beta}$ contains a set $\Pc_R(\gamma)$ with $|\Pc_R(\gamma)|\sim (R^{\frac23-\beta})^{6/p}$ such that 
		\begin{equation}
			\label{rokfitugijuhiou}
			\frac1{R^3}\int_{[0,R]^3}|\sum_{\xi\in\Pc_R(\gamma)}a_\xi e(\xi\cdot x)|^pdx\lesssim \|a_\xi\|_{l^2}^p\le |\Pc_R(\gamma)|^{\frac{p}{2}-1}\|a_\xi\|_{l^p}^p.
		\end{equation}
		Consider a partition of $\Gamma^3$ into arcs $\gamma$ of length $R^{-\beta}$.
		If we let $\Pc_R=\cup_{\gamma}\Pc_R(\gamma)$, then $|\Pc_R|\sim R^{6/p}$.
		
		The key new ingredient is the following small cap decoupling for $\Gamma^3$ proved in \cite{GM}. Let $r=R^{\max\{1,2\beta\}}$. For each $\beta\in(\frac13,1]$ and each complex measure $\mu$ supported on $\Gamma^3$ we have
		$$\|\widehat{\mu}\|_{L^p([0,r]^3)}^p\les R^{\beta(\frac{p}{2}-1)}\sum_{\gamma}\|\widehat{\mu_\gamma}\|^p_{L^p([0,r]^3)}.$$
		Since in our case $2\beta\le 1$, we have $r=R$. Combining this with \eqref{rokfitugijuhiou} concludes the proof of \eqref{ckljfhreughrtui0h0h0-} in this case. 
		
	\end{proof}
	It seems difficult to cover the full range $\frac{d^2+d+2}{2}<p<d(d+1)$ for \eqref{ckljfhreughrtui0h0h0-}  using decoupling methods, when $d\ge3$. When $d=3$, the  Guth--Maldague small cap decoupling holds in the range $p\in [8,12)$, but it is inefficient for us in the range $[8,10)$, since the spatial cube demanded by their decoupling is larger than $[0,R]^3$. The range $p\in [7,8)$ seems even more difficult to tackle.  Things are even more dramatic in higher dimensions, where small cap results  for the moment curve are rather sparse. 
	\smallskip
	
	As for hypersurfaces $\Sigma$, the $l^p(L^p)$ decoupling proved in \cite{BD2} may be combined with the methods in this paper to prove \eqref{ckljfhreughrtui0h0h0-} for $p\ge \frac{2(d+1)}{d-1}$, when $\Sigma$ has non-zero Gaussian curvature. This includes the case of the hyperbolic paraboloid in all dimensions.  To prove \eqref{ckljfhreughrtui0h0h0-} in the remaining range would require a sharp small cap decoupling theory for hypersurfaces in the range $\frac{2d}{d-1}\le p< \frac{2(d+1)}{d-1}$. Such results have recently been proved in \cite{GMO2} in the incomplete range $\frac{2d+4}{d}\le p< \frac{2(d+1)}{d-1}$. While small cap decoupling at the endpoint $p=\frac{2d}{d-1}$ is known to be equivalent with the Restriction Conjecture for $\Sigma$, it is an interesting question as to whether the validity of \eqref{ckljfhreughrtui0h0h0-} for $\Sigma$ with $p=\frac{2d}{d-1}$ would also imply the Restriction Conjecture.
	
	\section{Probabilistic estimates}
	\label{prob}
	In order to provide a better context to the results in this section, we start by recalling the classical Chernoff-Hoeffding estimate for random selectors. Let $(X_i)_{i=1}^N$ be $\{0,1\}$-valued i.i.d. random variables on a probability space, taking value 1 with probability $\delta$. Write $X=X_1+\ldots+X_M$, and note that $\E(X)=M\delta$. Then there is $c>0$ such that for each $M\ge \delta^{-1}$ we have
	\begin{equation}
		\label{CH}
		\P(X\in [M\delta/2,3M\delta/2])\ge 1-e^{-cM\delta}.
	\end{equation}
	We prove two theorems. 
	
	In order to be able to use the results in \cite{Tal_book}, we will work with real-valued orthonormal systems. The extension of these theorems to the case of complex exponentials is elementary, and presented at the end of this section, see Remark \ref{refinal}.  
	
	To be fully consistent with working in the framework of real Banach spaces, the constants $K(\Phi)$ in this section are defined using only real coefficients in \eqref{realvscomplex}. This however is ultimately irrelevant, as the constants defined with real versus complex coefficients are comparable to each other.
	
	\begin{te}
		\label{tnew}
		Let $\Phi=\{\varphi_1,\ldots,\varphi_N\}$ be a uniformly bounded orthonormal system consisting of real-valued functions,  and let $2\le q\le  p<\infty$.\\
		(a) Assume $p>q$. Then there exists $\Psi\subset \Phi$ satisfying 
		$$K_p(\Psi)\lesssim 1\text{  and  }|\Psi|\sim K_q(\Phi)^{-2q/p}|\Phi|^{q/p}.$$ 
		The implicit constants depend on $p,q$, but not on $\Phi$.
		\\
		(b) $(p=q)$ Assume that there exists $\kappa > 0$ such that $|\Phi|^{\kappa}\leq K_q (\Phi)^2$. Then there exists $\Psi\subset \Phi$ satisfying 
		$$K_q(\Psi)\lesssim 1\text{  and  }|\Psi|\sim K_q(\Phi)^{-2}|\Phi|.$$
		The implicit constants  depend on $q$ and $\kappa$, but not on $\Phi$.
	\end{te}
	
	Theorem \ref{tB} was generalized to Lorentz spaces by Talagrand \cite{Tal}. The result (a) was  proved by Bourgain \cite{Bo2}, but it can be also proved by Talagrand's argument. Here, we use Talagrand's argument in \cite{Tal_book} where his argument was simplified. The idea in the proof of Theorem \ref{tnew} is similar to \cite{Ryou}. 
	\smallskip
	
	The same approach can be used to prove the following strengthening of Theorem \ref{tnew}, that we record for possible future applications. 
	\begin{te}
		\label{mir2}$\Phi=\{\varphi_1,\ldots,\varphi_N\}$ be a uniformly bounded orthonormal system consisting of real-valued functions, and let $2\le q\le  p<\infty$.	 Let $(\delta_i)_{i=1}^N$ be $\{0,1\}$-valued i.i.d. random variables on the probability space $\Omega$ taking value 1 with probability $K_q(\Phi)^{-\frac{2q}{p}} |\Phi|^{-1+\frac{q}{p}}$. For $\omega\in\Omega$ we write 
		$$\Psi(\omega)=\{\varphi_i:\;\delta_i(\omega)=1\}.$$
		\\
		(a) If $p >q$, we have
		$$\P(K_p (\Psi(w)) \geq u)  \lesssim_{p,q} \exp (-C_{p,q} u^2 ),$$
		with $C_{p,q}>0$ depending on $p$ and $q$ but not on $N$ and $\Phi$.
		\\
		(b) $(p=q)$ Assume that there exists $\kappa >0$ such that $|\Phi|^{\kappa}  \leq K_{q} (\Phi)^2$. Then, we have
		$$\P\left(K_q(\Psi(\omega))\ge u\right)\lesssim_{q,\kappa} \exp (-C_{q} u^2 ),$$
		with the implicit constants depending on  $q$ and $\kappa$ but not on $\Phi$.
	\end{te}
	In our applications we only use (a) in Theorem \ref{mir2} with $q=2$. We remark that \cite{Bo} only proves the expected value
	\begin{equation}
		\label{4rioiti5i9yi}
		\E_\omega(K_p(\Psi(\omega))^p)\lesssim_p 1.\end{equation}
	This is weaker than the exponential estimate in (a), and proves insufficient for our application in the proof of Proposition \ref{p3}.
	\medskip

	We start by recalling the setup in \cite{Tal_book}. Let $(X,\|\cdot\|_X)$ be a real Banach space. Its dual $X^*$ is a  Banach space equipped with the canonical dual norm
	$$\|x^*\|_{X^*}=\sup_{\|x\|_X\le 1}|x^*(x)|.$$
	Alternatively, we may recover the norm on $X$ from its dual norm via the identity
	$$\|x\|_X=\sup_{\|x^*\|_{X^*}\le 1}|x^*(x)|.$$
	The closed unit ball of $(X,\|\cdot\|_X)$
	$$X_1=\{x\in X:\;\|x\|_{X}\le 1\}$$
	is a convex set with the property that for each $x\in X\setminus\{0\}$
	$$\{\lambda\in\R:\;\lambda x\subset X_1\}$$
	is a compact interval of positive length, centered at the origin. Conversely, given any $B\subset X$ that is convex and such that $\{\lambda\in\R:\;\lambda x\subset B\}$ is a compact, symmetric interval of positive length for each $x\in X\setminus \{0\}$, the Minkowski functional of $B$
	$$p_B(x)=\inf\{\alpha>0:\;x\in\alpha B\}$$
	defines a norm on $X$ whose closed unit ball is $B$.
	\smallskip
	
	We use this strategy to introduce a family of alternative norms on $X^*$ as follows. Fix some distinct  $x_1,\ldots,x_N\in X$. Then for each $C>0$, we consider the set 
	$$X_{1,C}^{*}=\{x^*\in X^*: \;\|x^*\|_{X^*}\le 1\text{ and }\sum_{n=1}^N|x^*(x_n)|^2\le C\}.$$
	This set can be seen to coincide with the closed unit ball of $X^*$ when equipped with the norm
	$$\|x^*\|_{X^{*}_C}=\inf\{\alpha>0:\;x^*\in \alpha X_{1,C}^*\}.$$This  is the dual norm of the norm on $X$ given by
	$$\|x\|_{X_C}=\sup\{|x^*(x)|:\;\|x^*\|_{X^*}\le 1\text{ and }\sum_{n=1}^N|x^*(x_n)|^2\le C\}.$$ 
	\smallskip
	
	We call $(X,\|\cdot \|_X)$ \textit{$ 2 $-convex} if there exists a constant $ \eta >0 $ such that 
	\begin{equation}\label{p_convex}
		\left\|\frac{x+y}{2}\right\|_X \leq 1- \eta \|x-y\|_X^2
	\end{equation}
	for any $ \|x\|_X,\|y\|_X \leq 1 $.
	
	In our application, we will consider $X=L^t([0,1]^d)$, for $t>2$. The 2-convexity of its dual space $X^*=L^{t'}([0,1]^d)$ was proved by Hanner \cite{Han} for $d=1$, but his proof can be easily extended to any probability measure space, in particular to $[0,1]^d$ for $d\ge 2$.

	For each $S\subset\{1,2,\ldots,N\}$ we define
	\begin{align*}
		\|U_{S}\|_C &=\sup_{x^* \in X_{1,C}^\ast} ( \sum_{n \in S} |x^{*}(x_n)|^2)^{1/2}\\&=\sup\{(\sum_{n\in S}|x^*(x_n)|^2)^{1/2}:\;\|x^*\|_{X^*}\le 1,\;\sum_{n=1}^N|x^*(x_n)|^2\le C\}.
	\end{align*}
	
	Let $(\delta_i)_{i=1}^N$ be $\{0,1\}$-valued i.i.d. random variables on the probability space $\Omega$ taking value $1$ with probability $\delta$. For $\omega \in \Omega$, we write
	\[
	\Psi(\omega) = \{ x_i : \delta_i(\omega) = 1 \}.
	\]

	Now we state the key result in  Talagrand's argument, as presented in \cite{Tal_book}. 
	\begin{te}\label{Tal_keythm}
		Consider a real Banach space $ X $ with a norm $ \|\cdot\|_X $ such that $ X^\ast $ is $ 2 $-convex with corresponding constant $ \eta $ in \eqref{p_convex}. Suppose that $ \|x_n\|_X \leq 1 $ for all $1 \leq  n \leq N$. Then, there exists a number $ K(\eta) $ that depends only on $ \eta $ with the following property. Consider a number $ C>0$ and define $ B = \max(C, K(\eta)\log N) $. Consider a number $ \delta>0 $ and assume that for some $ \epsilon >0 $,
		\[ \delta \leq \frac{1}{B\epsilon N^\epsilon } \leq 1. \]
		Given the random variables $ (\delta_i)_{1\leq i \leq N} $  defined above, we have
		\begin{equation}
			\label{rjirjgiigoti}
			\E_\omega(\|U_{\Psi(\omega)}\|_C^2) \lesssim \frac{K(\eta)}{\epsilon}. 
		\end{equation}
		Moreover, for $v \gtrsim 1$, there exists a constant $L$ independent of $\epsilon,v,C$,  such that
		\begin{equation}
			\label{rjirjgiigoti2}
			\P\left( \|U_{\Psi(\omega)}\|^2_C \gtrsim \frac{v K(\eta)}{\epsilon }  \right) \leq L \exp(-v/L).
		\end{equation}
	\end{te}
	Inequality \eqref{rjirjgiigoti} is proved in Theorem 19.3.1 in \cite{Tal_book}. Inequality \eqref{rjirjgiigoti2} is a strengthening of \eqref{rjirjgiigoti}. It follows by combining  Theorem 19.3.3  and (19.124) in \cite{Tal_book}. 
	\medskip
	
	We first prove  Theorem \ref{tnew}.
	Let $p$, $q$ and $\Phi$ be as in Theorem \ref{tnew}. There is nothing to prove if $p=2$, so we may assume $p>2$.

	We fix some arbitrary  $p_1 >p$ and  apply this theorem to  $X = L^{p_1}([0,1]^d)$ and $x_n=\varphi_n$. The value of $p_1$ is irrelevant, we may for example work with $p_1=3p/2$. While implicit constants in various computations will depend on $p_1$, by making the choice of $p_1$ explicit, they will in fact depend on $p$.
	\begin{co}\label{Cor_prob_UJ}
		Let
		\[
		\delta = K_q(\Phi)^{-\frac{2q}{p}} |\Phi|^{-1+\frac{q}{p}} \qquad  \text{and}  \qquad C_2 =\frac{K(\eta)}{\delta}\log N.    
		\]
		If $p>q$ we let
		$$C_1 = \frac{1}{\delta} N^{-\frac{1}{2}(1-\frac{q}{p})}.$$
		If $q=p$, we assume that there exist $\kappa >0$ such that
		\[
		N^\kappa  \le K_q(\Phi)^2,
		\]
		and we let
		\[
		C_1= \frac{1}{\delta} N^{-\kappa/2}.
		\]
		Then we have the estimates
		\begin{equation*}\label{EU_JC_1}
			\E(\norm{U_{\Psi(\omega)}}_{C_1}^2 ) \lesssim_{p_1,p,q}  \begin{cases}
				1 & \text{if} \ q>p
				\\
				\kappa^{-1} & \text{if} \ q=p
			\end{cases}
		\end{equation*}
		and
		\begin{equation*}\label{EU_JC_2}
			\E(\norm{U_{\Psi(\omega)}}_{C_2}^2 ) \lesssim_{p_1} \log N.
		\end{equation*}
	\end{co}

	\begin{proof}
		Note that $\eta$ and $K(\eta)$ are constants that only depend on $p_1$. 
		For sufficiently large $N$ we have
		\[
		B_1 = \max( C_1, K(\eta) \log N) = C_1 \qquad \text{and} \qquad B_2 = \max (C_2, K(\eta) \log N )= C_2.
		\]
		Let $\epsilon_1  = 1/2 (1-q/p)$ if $q >p$ and $\epsilon_1 =\kappa/2 $ if $q=p$ and let $\epsilon_2 = (100K(\eta)\log N)^{-1}$. Then
		\[
		\delta  \leq \frac{1}{B_1 \epsilon_1  N^{\epsilon_1} } \qquad \text{and} \qquad \delta   \leq \frac{1}{B_2 \epsilon_2  N^{\epsilon_2} }.
		\]
		By \eqref{rjirjgiigoti} in Theorem \ref{Tal_keythm}, we obtain the desired estimates.		
	\end{proof}
	\bigskip
	
	Following the recipe presented at the beginning of this section, we introduce the family of norms on $X^{*}=L^{p_1'}([0,1])$, whose closed unit balls are 
	$$X_{1,C}^*=\{g\in L^{p_1'}([0,1]^d):\;\|g\|_{p_1'}\le 1\text{ and }\sum_{n=1}^N|\int_{[0,1]^d}\varphi_n(x){g(x)}dx|^2\le C\},$$
	and the corresponding family of norms on $X=L^{p_1}([0,1]^d)$  
	\[
	\|f\|_C =  \sup_{g \in X_{1,C}^\ast } |\int_{[0,1]^d} f(x) {g(x)}dx|.
	\]
	
	Let us see the connection between these norms and the quantities $\|U_S\|_{C}$ introduced earlier. 
	\begin{lem}
		\label{lcon}	
		Assume $f= \sum_{n \in S}a_n \varphi_n $ such that $\sum_{n \in S} |a_n|^2 \leq 1$. Then, we have $\|f\|_C\le \|U_S\|_C$.
	\end{lem}
	\begin{proof}
		\begin{equation*}
			\begin{split}
				\norm{f}_{C} &= \sup_{g \in X_{1,C}^\ast} |\int f(x) {g(x)}dx|  =\sup_{g \in X_{1,C}^\ast} |\sum_{n \in S} a_n {\int \varphi_n(x) {g(x)} dx }|\\
				&\leq \left( \sum_{n \in S} |a_n|^2 \right)^{1/2}\left( \sup_{g \in X_{1,C}^\ast} \sum_{n \in S} \left|\int \varphi_n(x) {g(x)} dx \right|^2 \right)^{1/2}\\& \le \norm{U_S}_C.
			\end{split}
		\end{equation*}
	\end{proof}

	The next lemma is similar to Lemma 19.3.10 in \cite{Tal_book}. It is the proof of this lemma that justifies our choice of $p_1>p$.
	\begin{lem}\label{f_C1C2}
		We consider the same Banach space $X$ and constants $C_1 $ and $C_2$ as in Corollary \ref{Cor_prob_UJ}. If  $\norm{f}_{C_1} \leq 1$ and $\norm{f}_{C_2} \leq \sqrt{\log N}$, then $\norm{f}_{L^p([0,1]^d)} \lesssim 1$. The implicit constant is independent of $N$.
	\end{lem}
	In the proof of Lemma \ref{f_C1C2}, we will use the following lemma which is a particular case of Lemma 19.3.11 in \cite{Tal_book}. 
	
	\begin{lem}\label{split_f}
		If   $f \in L^{p_1}([0,1]^d)$ and $\norm{f}_C \leq 1$, then $f \in \mathcal{C}:= \mathrm{conv}(\mathcal{C}_1 \cup \mathcal{C}_{2})$ where $\mathcal{C}_1 = \{ g : \norm{g}_{L^{p_1}([0,1]^d)} \leq 1\}$ and $\mathcal{C}_{2} = \{ \sum_{1 \leq n \leq N} \beta_n \varphi_n : \sum_{1 \leq n \leq N} |\beta_n|^2 \leq C^{-1} \} $.
	\end{lem}
	
	Let us now prove Lemma \ref{f_C1C2}.
	\begin{proof}[Proof of Lemma \ref{f_C1C2}]
		Note that 
		\[ \norm{f}_{L^p([0,1]^d)}^p= p \int_0^\infty t^{p-1} \mu(\{|f| \geq t \} )dt  \]
		where $\mu$ is the Lebesgue measure restricted to $[0,1]^d$. We choose a constant $D$ such that $D^{p-2} =C_1 $ (this is possible since $p\ge q>2$) and split $\norm{f}_{L^p([0,1]^d)}^p$ according to the range of $t$ as follows
		\[
		\norm{f}_{L^p([0,1]^d)}^p \approx \int_0^1 t^{p-1} \mu(\{|f| \geq t \} )dt + \int_1^D t^{p-1} \mu(\{|f| \geq t \} )dt + \int_D^\infty t^{p-1} \mu(\{|f| \geq t \} )dt.
		\]
		It is easy to see that 
		\[
		\int_0^1 t^{p-1} \mu(\{|f| \geq t \} )dt \lesssim 1.
		\]
		We turn to the second term. Since $\norm{f}_{C_1} \leq 1$, Lemma \ref{split_f} implies that $f =u_1 +u_2$ where $\norm{u_1}_{L^{p_1}([0,1]^d)} \leq 1$ and $u_2 = \sum_{1 \leq n \leq N} \beta_{1,n}\varphi_n$ such that $\sum_{1 \leq n \leq N} |\beta_{1,n}|^2 \leq C_1^{-1}$. In particular, we have $\norm{u_2}^2_{L^{2}([0,1]^d)} \leq C_1^{-1}$. Thus
		\begin{equation*}
			\begin{split}
				\int_1^D t^{p-1}\mu(\{ |f| \geq t \} ) dt &\le \int_1^D t^{p-1}\mu(\{ |u_1| \geq t/2 \} ) dt +\int_1^D t^{p-1}\mu(\{ |u_2| \geq t/2 \} ) dt\\
				&\leq  \int_1^D t^{p-p_1-1} dt + C_1^{-1} \int_1^D  t^{p-3}dt \lesssim 1.
			\end{split}
		\end{equation*}
		Lastly, since $\norm{f}_{C_2} \leq \sqrt{\log N}$, Lemma \ref{split_f} imply that $f = v_1 +v_2$ such that $\norm{v_1}_{L^{p_1}([0,1]^d)} \leq \sqrt{\log N}$ and  $v_2 = \sum_{1 \leq n \leq N} \beta_{2,n}\varphi_n $ such that $\sum_{1 \leq n \leq N} |\beta_{2, n}|^2 \leq C_2^{-1}\log N  $. Similarly, 
		\begin{equation*}
			\int_D^\infty t^{p-1}\mu (\{ |f| \geq t \} ) dt \le \int_D^\infty t^{p-1} \mu(\{ |v_1| \geq t/2 \} ) dt +\int_D^\infty t^{p-1}\mu(\{ |v_2| \geq t/2 \} ) dt.
		\end{equation*}

		Note that $C_1 \gtrsim N^{c}$,  where the constant $c$ depends on $p$, $q$ and $\kappa$. If $p=q$, then we may take $c= \kappa/2$. Therefore, we get
		\[
		\int_D^\infty t^{p-1}\mu(\{ |v_1| \geq t/2 \} ) dt  \lesssim (\log N)^{p_1/2} \int_D^\infty t^{p-p_1-1} dt \lesssim 1.
		\]
		We find
		\[
		\norm{v_2}_{L^q([0,1]^d)} \lesssim K_q(\Phi)\left(\sum_n |\beta_{2,n}|^2 \right)^{1/2} \lesssim K_q(\Phi) (\log N )^{1/2} C_2^{-1/2}.
		\]
		Also, we have that
		\[
		\norm{v_2}_{L^\infty([0,1]^d)} \leq N^{1/2} \left(\sum_{n} |\beta_{2,n}|^2\right)^{1/2}\lesssim N^{1/2}(\log N )^{1/2} C_2^{-1/2}.
		\]
		Since
		\[
		C_2^{-1}\log N \lesssim K_q(\Phi)^{-\frac{2q}{p}} N^{1- \frac{q}{p}},
		\]
		we interpolate the inequalities above to get
		$$\norm{v_2}_{L^p([0,1]^d)}\le \norm{v_2}_{L^q([0,1]^d)}^{q/p}\norm{v_2}_{L^\infty([0,1]^d)}^{1-q/p}\lesssim 1.$$
		Finally, 
		\[
		\int_D^\infty t^{p-1}\mu(\{ |v_2| \geq t/2 \} ) dt  \lesssim \norm{v_2}_{L^p([0,1]^d)}^p \lesssim 1.
		\]
		We conclude that $\norm{f}_p \lesssim 1$.
		
	\end{proof}
	\begin{co}
		\label{rjutgurtu}	
		For each set $S\subset \{1,\ldots,N\}$ we have
		$$K_p(S)\lesssim 1+\|U_S\|_{C_1}+\frac{\|U_S\|_{C_2}}{\sqrt{\log N}}.$$
		The implicit constant is independent of $N$.
	\end{co}
	\begin{proof}
		Let $f(x) = \sum_{n \in S}a_n \varphi_n $ such that $\sum_{n \in S} |a_n|^2 \leq 1$. Lemma \ref{lcon} shows that
		$\norm{f}_{C}\le \norm{U_{S}}_C.$
		Consider 
		\[\Tilde{f}(x)=
		\frac{f(x)}{ 1+ \norm{U_S}_{C_1} + (\norm{U_{S}}_{C_2} / \sqrt{\log N})}.
		\]
		Since $\norm{\Tilde f}_{C_1} \leq 1$ and $\norm{\Tilde f}_{C_2} \leq \sqrt{\log N}$, 
		we may apply Lemma \ref{f_C1C2} to $\Tilde{f}$ and conclude that $\|\Tilde f\|_{p}\lesssim 1$.
		
	\end{proof}	
	\medskip
	
	Now, we are ready to prove Theorem \ref{tnew}.\begin{proof}[Proof of Theorem \ref{tnew}]
		
		We invoke Corollary \ref{rjutgurtu}
		to write 
		$$K_p(\Psi(\omega))\lesssim 1+\|U_{\Psi(\omega)}\|_{C_1}+\frac{\|U_{\Psi(\omega)}\|_{C_2}}{\sqrt{\log N}}.$$
		Corollary  \ref{Cor_prob_UJ} and H\"older implies that $\E(K_p(\Psi)) \lesssim_{p,q} 1$ if $p >q$ and $\E(K_p(\Psi)) \lesssim \kappa^{-1/2}$ if $p=q$. Since $|\Psi(\omega)|\sim K_q(\Phi)^{-2q/p} |\Phi|^{q/p} $ with probability close to 1, there is $\Psi \subset \Phi$ with desired properties.
		
	\end{proof}
	Let us turn to the proof of Theorem \ref{mir2}. 
	\begin{co}\label{Prop_est_UC}
		Let $\delta$, $C_1$ and $C_2$ be the constants in Corollary \ref{Cor_prob_UJ}. For sufficiently large $v$, we have the estimates
		\begin{equation*}
			\P\left(\norm{U_\Psi}_{C_1}^2 \geq v \right) \lesssim_{p,q}  \exp(-C_{p,q} v) \qquad \text{if} \ q >p,
		\end{equation*}
		\begin{equation*}
			\P\left(\norm{U_\Psi}_{C_1}^2 \geq \frac{v}{\kappa} \right) \lesssim_{q}   \exp(-C_{q} v) \qquad \text{if} \ q =p,
		\end{equation*}
		and
		\begin{equation*}
			\P(\norm{U_\Psi}_{C_2}^2 \geq \log(N) v ) \lesssim_{p,q} \exp(-C_{p,q} v).
		\end{equation*}
	\end{co}
	\begin{proof}
		The proof is similar to the proof of Corollary \ref{Cor_prob_UJ}. For sufficiently large $N$, we have $B_1 = C_1$ and $B_2 = C_2$. We let $\epsilon_1  = 1/2 (1-q/p)$ if $q >p$ and $\epsilon_1 =\kappa/2 $ if $q=p$ and let $\epsilon_2 = (100K(\eta)\log N)^{-1}$. Then, \eqref{rjirjgiigoti2} in Theorem \ref{Tal_keythm} implies the desired inequalities.
		
	\end{proof}
	We can now prove Theorem \ref{mir2} using a very similar argument.
	\begin{proof}[Proof of Theorem \ref{mir2}]
		By Corollary \ref{rjutgurtu}, we can write 
		$$K_p(\Psi(\omega))\lesssim 1+\|U_{\Psi(\omega)}\|_{C_1}+\frac{\|U_{\Psi(\omega)}\|_{C_2}}{\sqrt{\log N}}.$$
		If $q >p$, there exists a constant $C$ such that
		$$
		\P(K_p (\Psi(\omega) \geq u) \leq \P(\norm{U_{\Psi(\omega)}} \geq Cu ) + \P(\norm{U_{\Psi(\omega)}} \geq C \log(N)u ).
		$$
		Thus, Corollary \ref{Prop_est_UC} implies that $\P(K_p (\Psi(\omega) \geq u)  \lesssim_{p,q} \exp (-C_{p,q} u^2 )$.
		
		If $q=p$, there exists a constant $C$ such that 
		$$
		\P\left(K_q (\Psi(\omega) \geq \frac{u}{\kappa} \right) \leq \P\left(\norm{U_{\Psi(\omega)}} \geq C \frac{u}{\kappa}  \right) + \P\left(\norm{U_{\Psi(\omega)}} \geq C \log(N) u  \right).
		$$
		Using Corollary \ref{Prop_est_UC} again, we obtain that $\P(K_q (\Psi(\omega) \geq u/\kappa)  \lesssim_{q} \exp (-C_{q} u^2 )$. 
		
	\end{proof}
	\begin{re}
		\label{refinal}	
		We presented our arguments in the context of real-valued systems, and with constants $K_p$ defined using real coefficients $a_n$. To emphasize this distinction, let us denote this constant by $K_p^{\R}$. Given a complex-valued system, we may define the associated constant $K^{\C}_{p}$ by allowing complex coefficients. 
		
		Let us  consider an orthonormal system $\Phi$
		of complex exponentials $$\phi_n(x)=e(s_n\cdot  x)=\cos(2\pi s_n\cdot x)+i\sin(2\pi s_n\cdot x),$$ with $s_n\in S\subset\R^d$. It is easy to see that this forces $s_n-s_{m}$ to have at least one integer coordinate, for each $n,m\in S$. 
		Thus,  the (appropriately normalized) systems $\Phi_1=\{\cos(2\pi s_n\cdot x):\;s_n\in S\}$ and $\Phi_2=\{\sin(2\pi s_n\cdot x):\;s_n\in S\}$
		are themselves orthonormal systems, consisting of real-valued functions. The arguments we presented show that for each $i\in\{1,2\}$, a random selection of a subset $\Phi_i(\omega)$ of $\Phi_i$ will satisfy the specific requirements with probability very close to 1. So, in fact, such a generic selection will simultaneously work for both systems. Since 
		$$K_p^{\C}(\Phi(\omega))\sim K_p^{\R}(\Phi_1(\omega))+K_p^{\R}(\Phi_2(\omega)),$$
		all the probabilistic results we presented extend to systems of complex exponentials. 
	\end{re}

\end{document}